\documentclass[12pt]{article}
\usepackage{amsmath,amssymb,amsthm, amsfonts}
\usepackage{hyperref}
\usepackage{graphicx}
\usepackage{url}
\usepackage[mathlines]{lineno}
\usepackage{dsfont} 
\usepackage{color}
\definecolor{red}{rgb}{1,0,0}

\definecolor{blue}{rgb}{0,0,1}

\definecolor{green}{rgb}{0,.6,0}

\usepackage{float}

\usepackage{tikz}

\tikzstyle{vertex} = [fill,shape=circle,node distance=80pt]
\tikzstyle{edge} = [fill,opacity=.5,fill opacity=.5,line cap=round, line join=round, line width=50pt]
\tikzstyle{elabel} =  [fill,shape=circle,node distance=30pt]

\pgfdeclarelayer{background}
\pgfsetlayers{background,main}

\newtheorem{thm}{Theorem}[section]
\newtheorem{cor}[thm]{Corollary}
\newtheorem{lem}[thm]{Lemma}
\newtheorem{prop}[thm]{Proposition}
\newtheorem{conj}[thm]{Conjecture}

\newtheorem{quest}[thm]{Question}

\theoremstyle{definition}

\theoremstyle{definition}

\theoremstyle{definition}

\theoremstyle{definition}

\theoremstyle{definition}

\newcommand{\bit}{\begin{itemize}}
\newcommand{\eit}{\end{itemize}}
\newcommand{\ben}{\begin{enumerate}}
\newcommand{\een}{\end{enumerate}}
\newcommand{\beq}{\begin{equation}}
\newcommand{\eeq}{\end{equation}}
\newcommand{\bea}{\begin{eqnarray}} 
\newcommand{\eea}{\end{eqnarray}}
\newcommand{\bpf}{\begin{proof}}
\newcommand{\epf}{\end{proof}\ms}
\newcommand{\bmt}{\begin{bmatrix}}
\newcommand{\emt}{\end{bmatrix}}
\newcommand{\ms}{\medskip}

\newcommand{\lc}{\left\lceil}
\newcommand{\rc}{\right\rceil}

\newcommand{\noi}{\noindent}
\newcommand{\ceil}[1]{\lc #1 \rc}

\newcommand{\beqs}{\begin{equation*}} 
\newcommand{\eeqs}{\end{equation*}}
\newcommand{\beas}{\begin{eqnarray*}}
\newcommand{\eeas}{\end{eqnarray*}}

\newcommand{\floor}[1]{\lfloor #1 \rfloor}

\title{An upper bound for the $k$-power domination number in $r$-uniform hypergraphs}
\author{Joseph S. Alameda \thanks{Dept.~of Mathematics, United States Naval Academy, Annapolis, MD, USA (alameda@usna.edu)} \and
Franklin Kenter \thanks{Dept.~of Mathematics, United States Naval Academy, Annapolis, MD, USA (kenter@usna.edu)} \and 
Karen Meagher \thanks{Dept.~of Mathematics, University of Regina, Regina, SK, CA (karen.meagher@uregina.ca)} \and 
Michael Young \thanks{Dept.~of Mathematics, Carnegie Mellon University, Pittsburgh, PA, USA (michaely@andrew.cmu.edu)}
}

\begin{document}

\maketitle

\begin{abstract}
Generalizing work on graphs, Chang and Roussel introduced $k$-power domination in hypergraphs and conjectured the upper bound for the $k$-power domination number for $r$-uniform hypergraphs on $n$ vertices was $\frac{n}{r+k}$. This upper bound was shown to be true for simple graphs ($r=2$) and it was further conjectured that only a family of hypergraphs, known as the squid hypergraphs, attained this upper bound. In this paper, the conjecture is proven to hold for hypergraphs with $r=3$ or $4$; but is shown to be false, by a counterexample, for $r\geq 7$. Furthermore, we show that the squid hypergraphs are not the only hypergraphs that attain the original upper bound. Finally, a new upper bound is proven for $r\geq 3$.
\end{abstract}

\noi {\bf Keywords} zero forcing, power domination, hypergraphs

\noi{\bf AMS subject classification} 05C57, 05C15, 05C50, 05C65

\section{Introduction}
Power domination was introduced by Haynes, Hedetniemi, Hedetniemi and Henning in \cite{haynes} to study the monitoring process of electrical power networks by placing as few measurement devices as possible. This was done by defining the power domination problem in graph theoretic terms. In particular, Phase Measurement Units (PMUs) were placed at a set of initial vertices and then certain rules were applied. These rules consisted of a domination step followed by the zero forcing process.

Zero forcing was first introduced in \cite{AIM} as a way to find upper bounds for the maximum nullity of real symmetric matrices whose nonzero off-diagonal entries are described by a graph. It was later discovered to be one of the processes in power domination. Zero forcing has since been generalized, and specifically in hypergraphs, there have been multiple generalizations. In this paper, the generalization introduced in \cite{chang2} as a process in $k$-power domination in hypergraphs will be used. Moreover, this generalization also is consistent with the definition of $k$-power domination in simple graphs which was introduced by Chang, Dorbec, Montassier, and Raspaud in \cite{chang1}.

\section{Preliminaries}\label{prelim}
A \textit{hypergraph}, $\mathcal{H}= (V(\mathcal{H}),E(\mathcal{H}))$, is a set of vertices $V(\mathcal{H})$ combined with a set of edges $E(\mathcal{H})$ such that $E(\mathcal{H})$ is a subset of the power set of $V(\mathcal{H})$. When it is obvious which hypergraph is being used, $V$ and $E$ will be written. A hypergraph $\mathcal{H}$ is said to be \textit{$r$-uniform} when each edge in $E$ has order $k$. The \textit{closed neighborhood} of a vertex $v \in V$, $N[v]$, is the set of vertices adjacent to $v$ and $v$ itself. The \textit{open neighborhood} of a vertex $v \in V$ is the set $N(v) = N[v]\setminus\{v\}$. The closed (or open) neighborhood of a set $S$ is the set $\bigcup_{v\in S}N[v]$ (or $\bigcup_{v\in S}N(v)$) and is denoted $N[S]$ (or $N(S)$).

 A set of vertices $D$ is known as a \textit{dominating set} of a hypergraph $\mathcal{H}$ if $\bigcup_{v\in D}N[v] = V$. The size of a minimum dominating set in a hypergraph $\mathcal{H}$ is denoted $\gamma(\mathcal{H})$ and is called the \textit{domination number}. In \cite{chang2}, Chang and Roussel defined a \textit{$k$-power dominating set} on a hypergraph $\mathcal{H}$, as a set $D \subseteq V$ which colors vertices in $V$ blue with respect to the following rules: 
 
 \begin{enumerate}
     \item (Domination step) The vertices in $N[D]$ are colored blue.

     \item (Forcing step) Given a blue vertex $v$, if there is a set of at most $k$ edges each incident with $v$ that contains all of $v$’s white neighbors, then these neighbors turn blue.
     
     \item If iteratively applying these rules results in all vertices in $\mathcal{H}$ becoming blue, then $D$ is a {\it $k$-power dominating set} of $\mathcal{H.}$

 \end{enumerate}

 The {\it $k$-power dominating number} of $\mathcal{H}$, denoted $\gamma_p^k(\mathcal{H})$, is the minimum cardinality of a $k$-power dominating set of $\mathcal{H}.$ Notice the following inequality observed by Chang et al. in \cite{chang1} still applies to $k$-power domination for hypergraphs. Given a hypergraph $\mathcal{H}$,
$$\gamma(\mathcal{H}) = \gamma_p^0(\mathcal{H}) \geq \gamma_p^1(\mathcal{H}) \geq \dots \geq \gamma_p^k(\mathcal{H}) \geq \gamma_p^{k+1}(\mathcal{H}) \geq \dots.$$

Define the {\it white degree} of a vertex $v$ with respect to a set $S$ in $\mathcal{H}$, denoted $deg_w(v,S)$, to be the minimum number of edges that cover $v$ and the white neighbors of $v$ when each vertex in $S$ is blue and each vertex in $V(\mathcal{H})-S$ is white. Denote the {\it maximum white degree} by $\Delta_w(S)$. With the definition of white degree, the color change rule for $k$-forcing in hypergraphs can be redefined: if $v$ is a vertex in a set of blue vertices $S$ of a hypergraph $\mathcal{H}$ with $deg_w(v,S)\leq k$, then change the color of the neighbors of $v$ to blue. If applying this rule results in all vertices in $\mathcal{H}$ being colored blue, then $S$ is a $k$-forcing set of $\mathcal{H}$. Therefore the definition of a $k$-power dominating set in \cite{chang2}, is equivalent to a set $D$ such that once every vertex in $D$ and its neighborhood are colored blue, $N[D]$ is a $k$-forcing set.

Although $k$-power domination on simple graphs has been, and continues to be, a well studied area in graph theory, little is known about upper bounds for the $k$-power domination number of  hypergraphs. In \cite{chang2}, an upper bound for $r$-uniform hypergraphs was conjectured and it was believed the squid hypergraphs were the only hypergraphs attaining this bound. A {\it squid hypergraph} can be defined from any $r$-uniform hypergraph as follows. For every vertex $v$, add $(k+1)+(r-2)$ vertices $v_1,v_2,\dots,v_{k+1},v_1',v_2',\dots,v_{r-2}'$ and $k+1$ edges $e_{v,i}=\{v,v_1',v_2',\dots,v_{r-2}',v_i\}$ for $1\leq i \leq k+1$. See Figure \ref{squid} for an example.

\begin{figure}[H]
\centering
\includegraphics[width=0.8\textwidth]{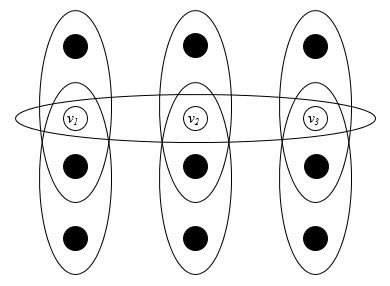}
\caption{A $3$-uniform squid hypergraph ($k=2$) constructed from the $3$-uniform hypergraph with a single edge: $\{v_1, v_2, v_3\}$.}
\label{squid}
\end{figure}

\begin{lem}[\cite{chang2}]\label{sizeofsquids}
Let $\mathcal{H}$ be a $r$-uniform hypergraph and $\mathcal{H}'$ be the squid hypergraph of $\mathcal{H}$. Then $\gamma_p^k(\mathcal{H}') = |\mathcal{H}|$.
\end{lem}

\begin{conj}[\cite{chang2}]\label{conj}
If $\mathcal{H}$ is a connected $r$-uniform hypergraph and $k+r\leq n$, then $$\gamma_p^k(\mathcal{H}) \leq \frac{n}{r+k} $$ with equality if and only if $\mathcal{H}$ is a squid hypergraph of a connected $r$-uniform hypergraph or $r=2$ with $\mathcal{H}=K_{k+2,k+2}.$
\end{conj}

In Section \ref{three}, it is shown that the first part of the Conjecture \ref{conj} holds for hypergraphs with $k+r \leq n$ vertices and $r\leq 4$. It is also shown that for $r\geq 3,$ $\gamma_p^k(\mathcal{H}) \leq \frac{n + \floor{\frac{r-3}{2}}m}{\floor{\frac{3(r-1)}{2}}+k}$. In Section \ref{four} a counterexample to Conjecture \ref{conj} is given for $k<r$ and $r\geq 7$. 

\section{An upper bound for $r$-uniform hypergraphs}\label{three}

In \cite{chang1}, Chang and Roussel proved that if $G$ is a connected graph of order $n\geq 2,$ then $\gamma_p^k(G) \leq \frac{n}{2+k}.$ Since simple graphs are $2$-uniform hypergraphs, Conjecture \ref{conj} holds for $r=2$. In this section, the upper bound in the conjecture is proven for $r\leq 4,$ and a new upper bound is proven for $r\geq 3.$

\begin{thm}[\cite{chang1}]
If $\mathcal{H}$ is a connected $2$-uniform hypergraph (simple graph) with $k+2\leq n$ vertices, then $$\gamma_p^k(\mathcal{H}) \leq \frac{n}{2+k}.$$
\end{thm}

The above theorem was proven using the following lemma from \cite{chang1}.

\begin{lem}[\cite{chang1}]
If $G$ is a connected graph with maximum degree at least $k+2$, then there exists a $k$-power dominating set containing only vertices of degree at least $k+2$.
\end{lem}

In order to prove Conjecture \ref{conj} for $r\leq 4$, the following lemma is needed.

\begin{lem}\label{first}
Let $\mathcal{H}$ be a connected hypergraph
  with $k+3 \leq n$ vertices. There exists a minimum $k$-power dominating set $D$ such that $k + 1 \leq deg_w(v, N[D \backslash \{v\}])$ for all $v \in  D$.

\end{lem}

\begin{proof}

    Given a set of vertices $S$, a vertex $v\in S$ is \textit{attached to $S$} if $|N(v)\cap S|\geq 1.$ Let $D$ be a $k$-power dominating set such that the number of attached vertices in $D$ is maximized. Moreover, choose $D$ such that out of all $k$-power dominating sets with maximum number of attached vertices, it has the minimum number of vertices $v$ with property  $deg_w(v,N[D\setminus\{v\}])\leq k$. We will also assume that $|D| \geq2$ since $|D| = 1$ is a trivial case.

  First note, if $v \in D$ and
  $deg_w(v, N[ D\backslash \{ v\} ] ) \leq k$, then for any vertex
  $u \in N(v)$ the set $\left( D \backslash \{v\} \right) \cup \{u\}$
  is also a minimum $k$-power dominating set. To see this  let $u\in N(v)$. In the domination step of $\left( D \backslash \{v\} \right) \cup \{u\}$, the set $N[D\backslash\{v\}]$ becomes blue and $u$ colors $v$ blue. Furthermore, $v$ can color its neighbors blue since $deg_w(v, N[ D\backslash \{ v\} ] ) \leq k$. Thus $N[D]$ becomes blue and $\left( D \backslash \{v\} \right) \cup \{u\}$ is a minimum $k$-power dominating set.

  Now assume that $v \in D$ is a vertex with
  $deg_w(v, N[ D\backslash \{ v\} ]) \leq k$. If there exists a
  $u \in N(v) \cap D$, then $D \backslash \{v\}$, by the
  note in the previous paragraph, is a $k$-power dominating set. But this is a contradiction
  since $D$ is a minimum set. So assume $v$ has no neighbors in $D$, and note this means $v$ is not attached.

  Since $\mathcal{H}$ is connected, there exists a path
  $v = u_0, u_1, \dots, u_\ell= v'$ to any vertex $v'$ in $D$. Pick
  $v'$ so that $v'$ is a vertex in $D$ closest to $v$; this implies
  that $u_i$ is not in $D$ for all $i \in \{1, \dots, \ell-1\}$, and
  that no $u_i$ is adjacent to a vertex in $D\backslash \{v\}$, except
  for $u_{\ell-1}$.
  
  Let $D_i= D\setminus \{v\}\cup \{u_i\}$ for $1\leq i \leq \ell$, and note that since $u_{\ell} = v' \in D$, $$|D_{\ell}| = |D\setminus \{v\}\cup \{u_{\ell}\}| = | D\setminus \{v\}\cup \{v'\}| = |D| - 1.$$ Furthermore, notice that for each $i$, $0\leq i \leq \ell-1$, if $D_i$ is a minimum $k$-power dominating set and $deg_w(u_i, N[D_i \backslash\{u_i\}]) \leq k$, then $D_{i+1}$ is a minimum $k$-power dominating set. However, $D_{\ell}$ is not a $k$-power dominating set because $|D_{\ell}| < |D|$ by definition. Therefore, there has to exist a first $u_j$ on the path (starting from $v$) with $deg_w(u_j, N[D_j \backslash\{u_j\}]) \geq k+1$, and $deg_w(u_i, N[D_i \backslash\{u_i\}]) \leq k$ for all $i<j$. Observe that $D_j$ is a $k$-power dominating set.

 If $j = \ell-1$, then $D_j$ is a minimum $k$-power dominating set that contradicts the maximality condition of $D$, since $u_j$ is adjacent to $v'\in D_j$.
  
  If $j=\ell -2$ and $y\in D_j$ has the property that $deg_w(y, N[D_j \backslash\{y\}]) \leq k < deg_w(y, N[D \backslash\{y\}])$, then there exists a vertex $x$ adjacent to $u_j$ and $y$ (if not, the path we chose was not minimum). Let $D_x = (D_j\setminus \{y\})\cup x$. Then, $D_x$ is a $k$-power dominating set that contradicts the maximality property of $D$ since $x$ is adjacent to $u_j$. If no vertex $y$ has the property that $deg_w(y, N[D_j \backslash\{y\}]) \leq k < deg_w(y, N[D \backslash\{y\}])$, then $D_j$ contradicts the minimality property of $D$.

  If $j\leq\ell-3,$ then $u_j$ is not adjacent to any neighbor of a vertex in $D$, or any vertex in $D$. Now, $D_j$ has one less vertex $v$ with property $deg_w(v,N[D_j\setminus \{v\}])\leq k$ compared to $D$ and both have the same number of attached vertices, a contradiction.
  
 Thus, no vertex $v\in D$ has the property $deg_w(v, N[ D\backslash \{ v\} ]) \leq k$.



\end{proof}

Given a set $S$ in a hypergraph in $\mathcal{H}$, the set of {\it external private neighbors} of a vertex $v\in S$, denoted $epn(v,S),$ is the set of vertices not in $S$ that are adjacent to $v$, but not to any other vertex in $S$.

\begin{lem}\label{upper}
Let $\mathcal{H}$ be a hypergraph and let $D$ be a minimum $k$-power dominating set of $\mathcal{H}$ such that $k+1\leq deg_w(v,N[D\setminus \{v\}])$. Then $k+1\leq deg_w(v,N[D\setminus \{v\}]) \leq |epn(v,D)|,$ for all $v\in D$.
\end{lem}

\begin{proof}
Using Lemma \ref{first}, let $D$ be a minimum $k$-power dominating set of $\mathcal{H}$ such that $k+1\leq deg_w(v,N[D\setminus \{v\}])$ for all $v\in D$. Color $v$ and $N[D\setminus \{v\}]$ blue, for some $v\in D$. Then the external private neighborhood of $v$ is left white, which implies $deg_w(v,N[D\setminus \{v\}])\leq epn(v,D)$. 
\end{proof}

In \cite{BHT}, Bujt\'as, Henning, and Tuza proved the following upper bound for dominating sets in hypergraphs.

\begin{thm}[\cite{BHT}] \label{BHT1}
If $\mathcal{H}$ is an $r$-uniform hypergraph of order $n$ with $m$ edges and no isolated vertices, and $r\geq 3$, then $$\gamma(\mathcal{H})\leq \frac{n + \floor{\frac{r-3}{2}}m}{\floor{\frac{3(r-1)}{2}}}.$$
\end{thm}

Setting $r=3,$ or $4$, results in the following corollary.

\begin{cor}[\cite{BHT}] \label{BHT2}
For any $r$-uniform hypergraph $\mathcal{H}$ on $n$ vertices with no isolated vertices, with $r=3$ or $4$, $$\gamma(\mathcal{H}) \leq \frac{n}{r}.$$
\end{cor}

With the Lemma \ref{first}, Lemma \ref{upper} and Theorem \ref{BHT1}, the following main result is proven.

\begin{thm}\label{winner}
If $\mathcal{H}$ is a connected, $r$-uniform hypergraph of order $n$, $k+r\leq n$, with $m$ edges, and $r\geq 3$, then $$\gamma_p^k(\mathcal{H}) \leq \frac{n + \floor{\frac{r-3}{2}}m}{\floor{\frac{3(r-1)}{2}}+k}.$$
\end{thm}

\begin{proof}

Let $\mathcal{H}$ be a connected, $r$-uniform hypergraph. By Lemma \ref{upper}, let $D$ be a minimum $k$-power dominating set such that for all $v\in D$, $k+1\leq deg_w(v,N[D\setminus\{v\}])\leq |epn(v,D)|$. 

For sake of contradiction, assume that $|V(\mathcal{H})| < \floor{\frac{3(r-1)}{2}}|D| - \floor{\frac{r-3}{2}}m +k|D|$. From the vertex set and edge set of $\mathcal{H}$, an auxiliary graph $\mathcal{H}'$ will be constructed by removing at least $k|D|$ vertices. This hypergraph will be constructed such that it will have a dominating set that dominates a $k$-forcing set of $\mathcal{H}$, which leads to a contradiction since this set will be a $k$-power dominating set in $\mathcal{H}$ that is smaller than $D$.

Let $\{v_1,v_2,\dots, v_{|D|}\}$ be an ordering of the vertices in $D$. For each $v_i$, let $P_{v_i}$ be a set of $k$ vertices from the external private neighborhood of $v_i$. Let $P$ be the union of each $P_{v_i}$ and let $B= N[D]\setminus P.$ Observe the $B$ is a $k$-forcing set for $\mathcal{H}$. 

To constuct $\mathcal{H}'$, each $P_{v_i}$ will be iteratively removed. First consider $\mathcal{H}\setminus P_{v_1}$. This hypergraph may have isolated vertices since removing vertices from $\mathcal{H}$ also removes edges with $r-1$ or fewer vertices. Let $I_1$ be the set of vertices that become isolated after removing $P_{v_1}$. For every vertex $x\in I_1$, pick exactly one vertex from $P_{v_1}$ that is in an edge with $x$ and $v_1$ in $\mathcal{H}$. Call this set of chosen vertices $T_1$. Let $|T_1|=t.$


In $\mathcal{H}$, if $deg_w(v_1,B\setminus I_{1})\leq k$, let $I_1=I_{v_1}$ and let $H_1=\mathcal{H}\setminus P_{v_1}$. If $deg_w(v_1,B\setminus I_{1})\geq k+1$, the following process is done on $\mathcal{H}\setminus P_{v_1}$ to construct $H_1$. Since $deg_w(v_1,B\setminus I_{1})\geq k+1$, then $deg_w(v_1,N[D]\setminus (I_1 \cup T_1))\geq t+1$. Since $deg_w(v_1,N[D]\setminus (I_1 \cup T_1))\geq t+1$, there exists at least $t+1$ vertices in $I_1$, each in a distinct edge containing $v_1$ and some vertex from $T_1$ in $\mathcal{H}$. Let $A_1$ be the set of these vertices in $I_1$ and let $A_2$ be the set of distinct edges containing these vertices. 

Let $x \in A_1$, $e’ \in A_2$ with $x \in e’$, and $C\subseteq A_1\setminus\{x\}$ with $|C|=|e'\cap T_1|$, then construct the new edge  $e= \{v_1, e'\setminus(e'\cap T_1),  C\}$ in $\mathcal{H}\setminus P_{v_1}$ and note that this edge contains $r$ vertices. Iteratively do this for every vertex in $A_1$. Once this process finishes, every vertex in $I_1$ is contained in one of the newly added edges. In this case, let $I_{v_1}= \{\emptyset\}$. Hence, in $\mathcal{H}$, $deg_w(v_1,B\setminus I_{v_1})\leq k$.  Let $H_1$ be the hypergraph constructed after this process is done on $\mathcal{H}\setminus P_{v_1}$.


The same process is then done on $H_1\setminus P_{v_2}$ to get $H_2$ and continued until $H_{|D|}$ is constructed. Let $I = \bigcup I_{v_i}$ and $\mathcal{H}' = H_{|D|}\setminus I.$ Let $m'$ be the number of edges in $\mathcal{H}'$ and note that $m'\leq m$. Observe that by construction $B\setminus I$ is a $k$-forcing set for $\mathcal{H}$.

Through this construction $|V(\mathcal{H}')| \leq |V(\mathcal{H})|-k|D| < \floor{\frac{3(r-1)}{2}}|D| - \floor{\frac{r-3}{2}}m$ therefore by Theorem \ref{BHT1}, 

\begin{align*}
    \gamma(\mathcal{H}') &< \frac{\floor{\frac{3(r-1)}{2}}|D| - \floor{\frac{r-3}{2}}m+\floor{\frac{r-3}{2}}m'}{\floor{\frac{3(r-1)}{2}}} \\
    &\leq  \frac{\floor{\frac{3(r-1)}{2}}|D| - \floor{\frac{r-3}{2}}m+\floor{\frac{r-3}{2}}m}{\floor{\frac{3(r-1)}{2}}} \\
    &= |D|.
\end{align*}

This is a contradiction since a minimum dominating set of $\mathcal{H}'$ dominates $B\setminus I$ in $\mathcal{H}$. Therefore, $$\gamma_p^k(\mathcal{H}) \leq \frac{n + \floor{\frac{r-3}{2}}m}{\floor{\frac{3(r-1)}{2}}+k}.$$

\end{proof}

Setting $r=3,$ or $4$, results in the following corollary.

\begin{cor}
For any connected $r$-uniform hypergraph $\mathcal{H}$ of order $n$, $k+r\leq n$, with no isolated vertices, with $r=3$ or $4$, $$\gamma_p^k(\mathcal{H}) \leq \frac{n}{r+k}.$$
\end{cor}

The squid hypergraphs show that the bound in Theorem \ref{winner} is tight for $r=3$ and $r=4$ by Lemma \ref{sizeofsquids}. 

The following construction from \cite{BHT} is used to show for any $r\geq 5$ there exists a $k$ such that $\frac{n + \floor{\frac{r-3}{2}}m}{\floor{\frac{3(r-1)}{2}}+k} = \gamma_p^k(\mathcal{H)}$. Let the hypergraph $\mathcal{H}$ be defined as follows. The vertex set of $\mathcal{H}$ is partitioned as $V=V_1\cup V_2 \cup V_3 \cup \{v_{12},v_{13},v_{23}\}$ where $|V_1|=\floor{\frac{r-1}{2}}$ and $|V_2|=|V_3|=\ceil{\frac{r-1}{2}}$. Let $V'_2$ be a subset of $V_2$ of size $\floor{\frac{r-1}{2}}$. Let the edge set of $\mathcal{H}$ consist of the edges $e_1= V_1\cup V_2 \cup \{v_{12}\}$, $e_2= V_1\cup V_3 \cup \{v_{13}\}$, and $e_3= V'_2\cup V_3 \cup \{v_{23}\}.$

Notice that with this construction, $\gamma_p^k(\mathcal{H)} = 1$ for $k\geq 1$ since in the domination step, the white neighbors of any blue vertex will be contained in a single edge. Setting $k= \floor{\frac{3(r-1)}{2}}$, 

    $$\frac{n + \floor{\frac{r-3}{2}}m}{\floor{\frac{3(r-1)}{2}}+k} = 1 = \gamma_p^k(\mathcal{H)}.$$

\section{Counterexamples}\label{four}

Although Conjecture \ref{conj} holds for $r\leq 4$, it is not true for $r\geq 7.$ A counterexample is now given for $r\geq7$. These hypergraphs are related to the hypergraphs constructed in \cite{BHT} that show the domination number upper bound is tight. Let $X=\{x_1,x_2,\dots,x_{k+1}\}$, $Y=\{y_1,y_2,\dots,y_{k+1}\}$, and $Z=\{z_1,z_2,\dots,z_{k+1}\}$, let $A_1,$ $A_2,$ and $A_3$ be sets of $k+2$ vertices, and let $B_1$ and $B_2$ be sets of $\ell$ vertices, $\ell\geq 0$. Let $\mathcal{H}$ be the hypergraph constructed with the following edge sets: 
$$A_1 \cup B_1 \cup \{x_i\} \cup A_2, \text{ } A_2 \cup B_1 \cup \{z_i\} \cup A_3 \text{ and } A_1 \cup B_2 \cup \{y_i\} \cup A_3, \text{ } 1 \leq i \leq k+1.$$ 

$\mathcal{H}$ is a $(5+2k+\ell)$-uniform hypergraph with $9+6k+2\ell$ vertices. If the conjecture is true, then $$\gamma_p^k(\mathcal{H}) \leq \frac{9+6k+2\ell}{5+3k+\ell} < 2.$$ Therefore $\gamma_p^k(\mathcal{H})=1.$ Suppose one vertex $v$ is chosen as a $k$-power dominating set. Then after the domination step, every vertex in $N(v)$ has white degree $k + 1$, and so $\{v\}$ is not a $k$-power dominating set.
 Therefore, $\gamma_p^k(\mathcal{H}) \geq 2$ which contradicts the conjecture.

In \cite{chang2}, Chang and Roussel conjectured that the only connected $r$-uniform hypergraphs, $\mathcal{H}$, where $\gamma_p^k(\mathcal{H}) = \frac{n}{r+k}$, were the squid hypergraphs. We will show that there are many other hypergraphs by extending the definition of squid hypergraphs.

 Given integers $d\geq 1,$ $k\geq1$ and $r\geq 2$ and a vector of positive integers $x=(x_1,\dots,x_d)$ where $1\leq x_i\leq r-1,$ construct a $(d,k,r,x)$-squid as follows. Let $$s_{1,1},\dots,s_{1,r-x_1},s_{2,1},\dots, s_{2,r-x_2},\dots,s_{d,1},\dots, s_{d,r-x_d}$$ be called the \textit{strong vertices} and let $$w_{1,1},\dots,w_{1,k+x_1},w_{2,1},\dots, w_{2,k+x_2},\dots,w_{d,1},\dots, w_{d,k+x_d}$$ be called the \textit{weak vertices}. Call the set of all vertices with first index $i$, the \textit{$i$-th spine}. Within each spine, add $k+1$ edges containing $r$ vertices between the strong vertices and weak vertices such that no subset of them cover the weak vertices in the spine. These edges will contain every vertex in the strong portion of the spine. Finally, add any edges as desired among the strong vertices (even amongst different spines), so long as the hypergraph is connected and $r$-uniform.

\begin{prop}
Let $\mathcal{H}$ be a $(d,k,r,x)$-squid with $d(r+k)=|V(\mathcal{H})|$, then $\gamma_p^k(\mathcal{H})=d$.
\end{prop}

\begin{proof}
To show that $\gamma_p^k(\mathcal{H})=d$, observe that any set that intersects the strong portion of each spine is a dominating set and hence a $k$-power dominating set. It suffices to show that any $k$-power dominating set must necessarily have one vertex within each spine. Suppose to the contrary that there is a $k$-power dominating set that does not intersect one of the spines. Since edges containing weak vertices must be contained with that spine, in order to color the weak vertices within that spine, one of the corresponding strong vertices must force. However, observe that since the spine requires $k+1$ edges to cover its weak vertices, no strong vertex can force. This completes the proof.
\end{proof}

\section{Conclusion}

We have shown that Conjecture \ref{conj} does not hold for $r\geq 7$, and have shown a new upper bound for the $k$-power domination number that is related to the upper bound of the domination number given in \cite{BHT}. The following questions still remain.

\begin{quest}
Does Conjecture \ref{conj} hold for $r=5$ or $6$?
\end{quest}

It is interesting to note that if $\frac{n}{r+k}$ is replaced with $\ceil{\frac{n}{r+k}}$ in Conjecture \ref{conj}, then that equation would hold for the counterexamples given in the previous section.

\begin{quest}
Is Theorem \ref{winner} tight for every $k$ and $r$?
\end{quest}

\section*{Acknowledgements}

We would like to thank the American Institute of Mathematics Zero forcing and its applications 2017 workshop. The second author was partially supported by the ONR with grant number N00014-19-W-X00094, and the NSF with grant number DMS-1719894. The third author was partially supported by the NSERC Discovery Research Grant with grant number RGPIN-03852-2018. The fourth author was partially supported by the NSF with award number 1719841.

\end{document}